\catcode`\ç\active\def ç{{\c c}}\catcode`\Ç\active\def Ç{{\c C}}
\catcode`\ð\active\def ð{{\u g}}\catcode`\Ð\active\def Ð{{\u G}}
\catcode`\ý\active\def ý{{\i}} \catcode`\Ý\active\def Ý{{\. I}}
\catcode`\ö\active\def ö{{\" o}} \catcode`\Ö\active\def Ö{{\" O}}
\catcode`\þ\active\def þ{{\c s}} \catcode`\Þ\active\def Þ{{\c S}}
\catcode`\ü\active\def ü{{\" u}} \catcode`\Ü\active\def Ü{{\" U}}

\documentclass{amsart}
\usepackage{amsmath}
  \usepackage{paralist}
  \usepackage{graphics} 
  \usepackage{epsfig} 
\usepackage{graphicx}  \usepackage{epstopdf}
 \usepackage[colorlinks=true]{hyperref}
\hypersetup{urlcolor=blue, citecolor=red}

\newtheorem{theorem}{Theorem}[section]

\newtheorem{lemma}[theorem]{Lemma}
\newtheorem{proposition}{Proposition}
\newtheorem{conjecture}{Conjecture}

\theoremstyle{definition}

\newtheorem{remark}{Remark}

\def\R {\mathbb{R}}

\title[Nonlinear Schrödinger equation on the half-line]
      {Qualitative properties of solutions for nonlinear Schrödinger equations with nonlinear boundary conditions on the half-line}

\author[Varga K. Kalantarov and Türker Özsarý ]{}

\subjclass[2010]{Primary: 35B44, 35B33, 35B40; Secondary: 93C20, 93D20, 93D15.}
\keywords{nonlinear Schrödinger equation, blow-up, exponential stabilization, critical exponent, nonlinear boundary condition, damping.}

\email{vkalantarov@ku.edu.tr}
\email{turkerozsari@iyte.edu.tr}

\begin{document}
\maketitle

\medskip
\centerline{\scshape Varga K. Kalantarov\textsuperscript{1} and Türker Özsari\textsuperscript{2},\footnote{Corresponding author.},\footnote{This author's research was supported by Izmir Institute of Technology under the BAP grant 2015ÝYTE43.}}
\medskip
{\footnotesize
 \centerline{}
\centerline{\textsuperscript{1} Department of Mathematics, Koç University}
   \centerline{Sarýyer, Ýstanbul 34450, TURKEY}
 \centerline{}
 \centerline{\textsuperscript{2} Department of Mathematics, Izmir Institute of Technology}
   \centerline{Urla, Ýzmir 35430, TURKEY}
} 

\bigskip


\begin{abstract}
In this paper, we study the interaction between a nonlinear focusing
Robin type boundary source, a nonlinear defocusing interior source,
and a weak damping term for nonlinear Schrödinger equations posed on
the infinite half line.  We construct solutions with negative initial
energy satisfying a certain set of conditions which blow-up in finite time
in the $H^1$-sense.  We obtain a sufficient condition relating the powers of
nonlinearities present in the model which allows construction of blow-up solutions.
In addition to the blow-up property, we also discuss the stabilization property and
the critical exponent for this model.
\end{abstract}
\section{Introduction}
In this paper, we consider the following nonlinear Schrödinger equation (NLS) model
posed on the infinite half line:
\begin{equation}\label{nlNeumannProb} \left\{ \begin{array}{ll}
         i\partial_t u - u_{xx}+k|u|^pu+iau= 0, & \mbox{$t>0$, $x\in I=(0,\infty)$},\\
         u(x,0)=u_0(x), & \mbox{$x>0$},\\
         u_x(0,t)= -\lambda|u(0,t)|^ru(0,t), & \mbox{$t>0$},\end{array} \right.
         \end{equation} where $u(x,t)$ is a complex valued function, the real variables
         $x$ and $t$ are space and time coordinates, and subscripts denote partial derivatives.  The constant parameters satisfy: $\lambda, p, k, r>0$ and $a\ge 0$.  When $\lambda=0$, the boundary condition reduces to the classical homogeneous Neumann boundary condition. When $r=0$, the boundary condition is the classical homogeneous Robin boundary condition.  When $\lambda$ and $r$ are both non-zero as in the present case, the boundary condition can be considered a nonlinear variation of the Robin boundary condition.

NLS is a classical field equation whose popularity increased especially when it was shown to be integrable in \cite{ZS}.  Although  it has many applications in physics, NLS does not model the evolution of a quantum state, unlike the linear Schrödinger equation.   Applications of NLS include transmission of light in nonlinear optical fibers and planar wavequides, small-amplitude gravity waves on the surface of deep inviscid water, and Langmuir waves in hot plasmas \cite{Sulem}, \cite{MB}.  NLS also appears as a universal equation governing the evolution of slowly varying packets of quasi-monochromatic waves in weakly nonlinear dispersive media \cite{Sulem}, \cite{MB}.  Some other interesting applications of NLS include Bose-Einstein condensates \cite{PS}, Davydov's alpha-helix solitons \cite{BR}, and plane-diffracted wave beams in the focusing regions of the ionosphere \cite{AVG}.

There is a large literature on the qualitative behavior of solutions for NLS. Our particular attention in this paper will be the blow-up and stabilization
of solutions at the energy level.  The blow-up theory for nonlinear Schrödinger equations in the presence of a damping term has attracted the attention of several scientists. Some of the major work in this subject are \cite{MT2}, \cite{GF1}, and \cite{OT}. Stabilization of solutions for weakly damped nonlinear Schrödinger equations has been studied well with homogeneous boundary conditions (see for example \cite{TM1}). Regarding nonhomogeneous boundary conditions; see \cite{Ozsari1}-\cite{Ozsari3} .

The model \eqref{nlNeumannProb} with linear main equation ($k=0$) and no damping ($a=0$) has been studied in \cite{AASKD}.  Local existence
and uniqueness of $H^1$ solutions have been obtained for
sufficiently smooth data ($u_0\in H^3(\mathbb{R}_+)$).  For those
local solutions, global existence of $H^1$ solutions has been
obtained for $r<2$ in the case of arbitrarily large data, and for
$r=2$ in the case of small data. It has also been shown that
solutions with strictly negative energy blow up if $r\ge2$ where the energy function is defined by
\begin{equation}\label{energy}E(t)\equiv
\|u_x(t)\|_{L^2(I)}^2-\frac{2\lambda}{r+2}|u(0,t)|^{r+2}+
\frac{2k}{p+2}\|u(t)\|_{L^{p+2}(I)}^{p+2}
\end{equation} for $t\ge
0$. Therefore, $r=2$ was considered to be the critical exponent for
the blow-up problem in the linear model.  There is another study (see \cite{LT1}) where the linear Schrödinger equation was considered with nonlinear boundary conditions. In \cite{LT1},
the authors obtain well-posedness and decay rate estimates at the
$L^2-$level for the Schrödinger equation with nonlinear, attractive, and dissipative boundary conditions of type $\frac{\partial u}{\partial \nu}=ig(u)$, where $g$ satisfies some monotonicity conditions.  Most recently, the nonlinear Schrödinger equation of cubic type was studied with nonlinear dynamical boundary conditions, which are equivalent to so called (nonlinear) Wentzell boundary conditions (see \cite{Wellington}).  However, this work also uses the fact that the structure of the given boundary condition provides a nice monotonicity, which helps to get a semigroup in an appropriate Sobolev space.  The nature of our model is very different than those in \cite{LT1} and \cite{Wellington} due to the lack of monotonicity, since in our case $\lambda$ is real.

Our first aim in this paper is to study the blow-up problem in a more general context than in \cite{AASKD}. In our model, the main equation also includes a nonlinear defocusing term ($k|u|^pu$, $k>0$) and damping ($iau,a\ge 0$).  In particular, we want to understand the nature of the competition between the bad term (nonlinear Robin boundary condition of focusing type) and the good terms (defocusing nonlinearity and damping).  We show that there are solutions which blow up in finite time. More precisely, we prove that solutions cannot exist globally in $H^1$ sense if the initial data and powers of nonlinearities satisfy a certain set of conditions.

The second aim of this paper is to obtain decay rate estimates. We will prove exponential stabilization of solutions where the decay rates are
determined according to the relation between the powers of the nonlinearities. We obtain different decay rates depending on the given relation between the powers of nonlinearities $r$ and $p$.

We comment on the critical exponent in the last chapter of the paper.  Recall that the critical exponent in the case $k=0, a=0$ is $r^*=2$ (see \cite{AASKD}).  However, in the presence of the defocusing nonlinearity, we deduce that the critical exponent must also depend on $p$.  For example, we show that every local solution is also global if $2\le r<\frac{p}{2}$ in Proposition \ref{stab2}.  This shows that sufficiently strong defocusing nonlinearity in the main equation has a dominating effect on the nonlinear boundary condition.
\begin{remark}\label{localassmp}
We do not study the local well-posedness of \eqref{nlNeumannProb}.  We assume that \eqref{nlNeumannProb} has a unique classical local solution on a maximal time interval $[0,T_{max})$ ($0<T_{max}\le \infty$), which lies in a Sobolev space of sufficiently high order and also satisfies the blow-up alternative in $H^1$ sense: either $T_{max}=\infty$ or else $T_{max}<\infty$ and $\|u_x(t)\|_{L^2(I)}\rightarrow \infty$ as $t\uparrow T_{max}$.  For
simplicity, we assume that the initial data is from $H^s(\R^+)$ with $s$ big enough and satisfies the necessary compatibility
condition that guarantees the existence of a local classical solution. Indeed, the second author's recent paper \cite{BO} proves the following local well-posedness theorem for the case $a=0$, but the proof can be trivially adapted to the case $a>0$.

\begin{theorem}[Local well-posedness]\label{LocalWellP} Let $T>0$ be arbitrary, $s\in \left(\frac{1}{2},\frac{7}{2}\right) -\left\{\frac{3}{2}\right\}$,  $p,r>0$, $k,\lambda\in\mathbb{R}-\{0\}$, $u_0\in H^s(\mathbb{R_+})$ together with ${u_0'(0)=-\lambda|u_0(0)|^ru_0(0)}$ whenever $s>\frac{3}{2}$.  We in addition assume the following restrictions on $p$ and $r$:
\begin{itemize}
  \item[(A1)]  If $s$ is integer, then $p\ge s$ if $p$ is an odd integer and $[p]\ge s-1$ if $p$ is non-integer.
  \item[(A2)]  If $s$ is non-integer, then $p>s$ if $p$ is an odd integer and $[p]\ge [s]$ if $p$ is non-integer.
  \item[(A3)] $r>\frac{2s-1}{4}$ if $r$ is an odd integer and $[r]\ge \left[\frac{2s-1}{4}\right]$ if $r$ is non-integer.
\end{itemize}  Then, the following hold true.
\begin{itemize}
  \item[(i)] Local Existence and Uniqueness: There exists a unique local solution $u\in X_{T_0}^s$ of \eqref{nlNeumannProb} for some $T_0=T_0\left(\|u_0\|_{H^s(\mathbb{R}_+)}\right)\in (0,T]$, where $X_{T_0}^s$ is the set of those elements in $$C([0,T_0];H^s(\mathbb{R}_+))\cap C(\mathbb{R}_+^x;H^{\frac{2s+1}{4}}(0,T_0))$$ that are bounded with respect to the norm ${\|\cdot\|_{X_{T_0}^s}}$.   This norm is defined by $$\|u\|_{X_{T_0}^s}:=\sup_{t\in[0,{T_0}]}\|u(\cdot,t)\|_{H^s(\mathbb{R_+})}+\sup_{x\in\mathbb{R}_+}\|u(x,\cdot)\|_{H^{\frac{2s+1}{4}}(0,{T_0})}.$$
  \item[(ii)] Continuous Dependence: If $B$ is a bounded subset of $H^s(\mathbb{R}_+)$, then there is $T_0>0$ (depends on the diameter of $B$) such that the flow $u_0\rightarrow u$ is Lipschitz continuous from $B$ into $X_{T_0}^s$.
  \item[(iii)] Blow-up Alternative: If $S$ is the set of all $T_0\in (0,T]$ such that there exists a unique local solution in $X_{T_0}^s$, then whenever $\displaystyle T_{max}:=\sup_{T_0\in S}T_0<T$, it must be true that ${\displaystyle\lim_{t\uparrow T_{max}}\|u(t)\|_{H^s(\mathbb{R}_+)}=\infty}$.
\end{itemize}
 \end{theorem}

\end{remark}
\section{Main Theorems}
Here are our main results.
\begin{theorem}[Blow-up]\label{MainThm01} Suppose $r> \max\{2,p-2\}$, $E(0)\le 0$,
and \begin{equation}\label{u0assmp}\frac{(a-b)}{2}\int_0^\infty
x^2|u_0(x)|^2dx< Im \int_0^\infty
xu_{0}(x)'\bar{u}_0(x)dx\end{equation} where
$b=\frac{a(r+2)(4-M)}{4(r+2)-2M}<0$, $M=\max\{8,2p\}$. Then, there
exists $T>0$ such that the corresponding local solution $u$ of
\eqref{nlNeumannProb} (see Remark \ref{localassmp}) satisfies
$$\lim_{t\rightarrow T^-}\|u_x(t)\|_{L^2(I)}=\infty.$$
\end{theorem}

\begin{remark}Note that in the case $a=0$, the assumption \eqref{u0assmp}
reduces to $$Im\int_0^\infty xu_{0}'\bar{u}_0dx>0.$$
This is the same assumption on the initial data in the context of the classical
paper \cite{G}. \end{remark}
\begin{remark}Note that we do not assume that the initial energy is strictly negative.  In the case $E(0)=0$, solutions do not have to blow-up if
one disregards \eqref{u0assmp}, e.g., the zero solution.  As we will see in the proof, the condition \eqref{u0assmp} forces solutions to blow-up in this case.  However, if one puts a stronger assumption on the initial energy, such as strict negativeness in the case $a=0$, we believe that by using a compactly supported weight function, see for example \cite{OgTs}, one might remove the condition \eqref{u0assmp} and still obtain the blow-up in $H^1$ sense.  \end{remark}
\begin{theorem}[Stabilization]\label{MainThm02} Suppose $u$ is a local solution of \eqref{nlNeumannProb} (see Remark \ref{localassmp}). Then we have the following:
\begin{enumerate}
  \item[(i)] if $a>0, r<2$, then $u$ is global and
 $$\|u(t)\|_{H^1(I)}^2\le Ce^{-(2a-\epsilon)t}, t\ge 0$$ where $\epsilon>0$
 is fixed and small (can be chosen arbitrarily small), and $C=C(u_0,\epsilon,r)$
 is a non-negative constant.
  \item [(ii)] if $a>0, 2\le r <\frac{p}{2}$, then $u$ is global and
 $$\|u(t)\|_{H^1(I)}^2\le Ce^{-(a\mu-\epsilon)t}, t\ge 0$$ where
 \begin{equation}\label{defmu}
 \mu=\frac{(p+2)(p-2r)}{p(p+2)-2r},
 \end{equation}
 and $\epsilon>0$ is fixed and small (can be chosen arbitrarily small),
  and $C=C(u_0,\epsilon,r,p)$ is a non-negative constant.
 \item[(iii)] if $a>0, r=2, p\le 4$, and $u_0$ is sufficiently small in $L^2$ sense,
 then $u$ is global and
 $$\|u(t)\|_{H^1(I)}^2\le Ce^{-2at}, t\ge 0,$$ where $C=C(u_0,p)$ is a non-negative
 constant.
 \item[(iv)] if $a>0, r>2, r\ge \frac{p}{2}$, and $u_0$ is sufficiently small
 in $H^1\cap L^{p+2}$ sense, then $u$ is global and
 $$\|u(t)\|_{H^1(I)}^2\le Ce^{-2at}, t\ge 0,$$ where $C=C(u_0,r,p)$ is a non-negative constant.

\end{enumerate}
\end{theorem}
\begin{remark}
The following problem remains open:
\begin{itemize}
  \item Is it possible to construct blow up solutions in the two cases
  $r=2,p\le4$ and $r>2,p-2\ge r\ge\frac{p}{2}$?
\end{itemize}
In our analysis, we show that this is not possible whenever we choose small
enough initial data.  However, this does not mean one cannot construct
blow-up solutions with arbitrary initial data.  An answer to the above problem
will also help to determine the critical exponent for our model, see Section \ref{SecCrt}.
\end{remark}
We summarize our results in the following table:
\begin{center}
\begin{table}[h]
\begin{tabular}{|l|l|l|l|}\hline
Nonlinear Powers                & Blow-up  & Local$\Rightarrow$Global  & Exp. Stabilization  \\
                 & ($a\ge0$)      &      ($a\ge0$)         &   ($a>0$)        \\\hline
r\textless2                     & NO      & YES              & YES           \\
                 &      &               &  Decay rate $\sim O(e^{-(2a-\epsilon)t})$         \\\hline
$2\le r< \frac{p}{2}$           & NO      & YES              & YES           \\
                 &      &               &  Decay rate $\sim O(e^{-(a\mu-\epsilon)t})$ (See \eqref{defmu})         \\\hline
$r=2,p\le4$                     &     & Small Sol.  & Small Sol.\\
                &      &               &  Decay rate $\sim O(e^{-2at})$         \\
                                   &   OPEN    &   Large Sol: OPEN             &  Large Sol: OPEN         \\\hline
$r>2,p-2\ge r\ge\frac{p}{2}$    &     & Small Sol.  & Small Sol.\\
                 &     &             &  Decay rate $\sim O(e^{-2at})$         \\
                  &   OPEN    &     Large Sol: OPEN             &  Large Sol:OPEN          \\\hline
$r>2,r>p-2$                     &     & ONLY Small Sol.  & ONLY Small Sol.\\
     &   &             &  Decay rate $\sim O(e^{-2at})$         \\
     &  YES   &             &         \\\hline
\end{tabular}
\caption{}
\end{table}
\end{center}

\section{Blow-up Solutions: Proof of Theorem \ref{MainThm01}}

\subsection{Case $a\neq 0$:}  In this section, we prove
Theorem \ref{MainThm01} for the case $a\neq 0$,  slightly  modifying
the proof in \cite{MT2}.

\begin{lemma}\label{massenergylemma} Let $u$ be a local solution of
\eqref{nlNeumannProb} (see Remark \ref{localassmp}) and $b\in \mathbb{R}$.
Then, \begin{enumerate}
\item[(i)] $\|u(t)\|_{L^2(I)}^2=e^{-2at}\|u_0\|_{L^2(I)}^2$,
 \item[(ii)] $E(t)e^{2bt}=E(0)+\int_0^te^{2bs}\rho(s)ds$
  \end{enumerate}  for $T_0>t \ge 0,$ where $\rho$ is given by
  \eqref{gs}, and $E(t)$ is defined in \eqref{energy}.
\end{lemma}
\begin{proof}
We multiply \eqref{nlNeumannProb} by $\bar{u}$, take the imaginary parts,
integrate over $I\equiv (0,\infty)$, and obtain the exponential decay of the
$L^2-$norm (conservation when $a=0$) of the solution.
$$\frac{1}{2}\frac{d}{dt}\|u(t)\|_{L^2(I)}^2=-a\|u(t)\|_{L^2(I)}^2\Rightarrow
\|u(t)\|_{L^2(I)}^2=e^{-2at}\|u_0(x)\|_{L^2(I)}^2.$$

Now, we multiply \eqref{nlNeumannProb} by $\bar{u}_t$, take two real
parts, integrate the obtained relation over $I$, and get
\begin{multline}\label{H1iden01}\frac{d}{dt}\left(\|u_x(t)\|_{L^2(I)}^2-
\frac{2\lambda}{r+2}|u(0,t)|^{r+2}+\frac{2k}{p+2}\|u(t)\|_{L^{p+2}(I)}^{p+2}\right)
\\= 2Re \int_0^\infty ia\bar{u}u_tdx= 2a Re \int_0^\infty \bar{u}(x,t)\left(u_{xx}-
k|u|^pu-iau\right)dx\\
=-2a\left(\|u_x(t)\|_{L^2(I)}^2-\lambda|u(0,t)|^{r+2}+k
\|u(t)\|_{L^{p+2}(I)}^{p+2}\right)\\
=-2a\left(\|u_x(t)\|^2-\frac{2\lambda}{r+2}|u(0,t)|^{r+2}+
\frac{2k}{p+2}\|u(t)\|_{L^{p+2}(I)}^{p+2}\right)\\
-\frac{2akp}{p+2}\|u(t)\|_{L^{p+2}(I)}^{p+2}+{\frac{2a\lambda
r}{r+2}|u(0,t)|^{r+2}}.
\end{multline}

Then,  the identity in \eqref{H1iden01} is simply
\begin{equation}\label{H1iden01b}
E'(t)=-2aE(t)-\frac{2akp}{p+2}\|u(t)\|_{L^{p+2}(I)}^{p+2}+
{\frac{2a\lambda r}{r+2}|u(0,t)|^{r+2}}.
\end{equation}
Adding $2bE(t)$ to both sides, where $b\in\mathbb{R}$ and $b<a$, we have
\begin{equation}\label{E2a2b01}E'(t)+2bE(t)=(2b-2a)E(t)-
\frac{2akp}{p+2}\|u(t)\|_{L^{p+2}(I)}^{p+2}+{\frac{2a\lambda r}{r+2}|u(0,t)|^{r+2}}.
\end{equation}
Rewriting the right hand side of \eqref{E2a2b01} by using the definition of $E(t)$,
we have
\begin{multline*}
E'(t)+2bE(t)=-(2a-2b)\|u_x(t)\|_{L^2(I)}^2-\frac{4\lambda
b}{r+2}|u(0,t)|^{r+2}\\
+ \frac{4kb}{p+2}\|u(t)\|_{L^{p+2}(I)}^{p+2}
+{2a\lambda}|u(0,t)|^{r+2}-{2ak}\|u(t)\|_{L^{p+2}(I)}^{p+2}.
\end{multline*}
Multiplying both sides by $e^{2bt}$ and integrating over $(0,t)$, we
have
\begin{equation*}
E(t)e^{2bt}=E(0)+\int_0^te^{2bs}\rho(s)ds,
\end{equation*} where
\begin{multline}\label{gs}
\rho(t)=-(2a-2b)\left(\|u_x(t)\|_{L^2(I)}^2-
\left(\frac{a(r+2)-2b}{2a-2b}\right)\frac{2\lambda}{r+2}|u(0,t)|^{r+2}\right.\\
\left.+\left(\frac{a(p+2)-2b}{2a-2b}\right)\frac{2k}{p+2}\|u(t)\|_{L^{p+2}(I)}^{p+2}\right).
\end{multline}
\end{proof}

Let us set
\begin{equation}\theta(t)\equiv
\|u_x(t)\|_{L^2(I)}^2-\left(\frac{a(r+2)-2b}{2a-2b}\right)
\frac{2\lambda}{r+2}|u(0,t)|^{r+2}+
\frac{2k}{p+2}\|u(t)\|_{L^{p+2}(I)}^{p+2}.
\end{equation}
Note that $\frac{a(r+2)-2b}{2a-2b}\ge 1$, which implies
$\theta(t)\le E(t)$. Therefore,
$\frac{a(p+2)-2b}{2a-2b}-1=\frac{ap}{2a-2b}>0$, and by Lemma
\ref{massenergylemma}
\begin{multline}\label{e11}\theta(t)e^{2bt}\le E(t)e^{2bt}=
 E(0)-(2a-2b)\int_0^t \left(\theta(s)+\left(\frac{ap}{2a-2b}\right)
 \frac{2k}{p+2}\|u(s)\|_{L^{p+2}(I)}^{p+2}\right)e^{2bs}ds\\
\le E(0)-(2a-2b)\int_0^t \theta(s)e^{2bs}ds.
\end{multline}

Multiplying \eqref{e11} by $e^{(2a-2b)t}$, we get
\begin{equation}\frac{d}{dt}\left( e^{(2a-2b)t}\int_0^te^{2bs}
\theta(s)ds\right)\le E(0)e^{(2a-2b)t}\end{equation}
from which it follows that \begin{equation}\label{e1negative}\int_0^t\theta(s)e^{2bs}ds
\le  0\end{equation} provided that $E(0)\le 0.$

Now, we set \begin{equation}I(t)=\int_0^\infty x^2|u|^2dx, V(t)=
-4Im\int_0^\infty \bar{u}xu_xdx,\text{ and }y(t)=-\frac{1}{4}V(t).\end{equation}
We have the following lemma.
\begin{lemma}\label{virial}
$I$ and $y$ satisfy the following identities:
\begin{enumerate}
  \item[(i)] $e^{2bt}I(t)+(2a-2b)\int_0^te^{2bs}I(s)ds=I(0)+\int_0^tV(s)e^{2bs}ds$,
  \item[(ii)] $\dot{y}+2ay = -\frac{1}{4}\theta_1$, and
  \item[(iii)] $V(t)e^{2bt}=V(0)+(2b-2a)\int_0^tV(s)e^{2bs}ds+\int_0^t\theta_1(s)^{2bs}ds$
\end{enumerate} for $T_0\ge t\ge 0$ where $\theta_1$ is given in \eqref{Vprime}.
\end{lemma}
\begin{proof}
Differentiating $I(t)$, we have
    \begin{multline}\label{dtI}
        \frac{d}{dt}I(t)=\int_0^\infty x^2(u\bar{u}_t+u_t\bar{u})dx=
        2Re\int_0^\infty x^2u_t\bar{u}dx
        \\=2Im\int_0^\infty (u_{xx}-k|u|^pu-iau)x^2\bar{u}dx=
        -2Im\int_0^\infty(x^2\bar{u})_xu_xdx-2a\int_0^\infty x^2|u|^2dx\\
        =-4Im\int_0^\infty \bar{u}xu_xdx-2a\int_0^\infty x^2|u|^2dx.
    \end{multline} Therefore,
    \begin{equation}I'(t)+2aI(t)=-4 Im \int_0^\infty \bar{u}xu_xdx.\end{equation}
    Adding $2bI(t)$ to both sides,
    \begin{equation}\label{Iest01}I'(t)+2bI(t)=-(2a-2b)I(t)+V(t).
    \end{equation}
    Multiplying both sides by $e^{2bt},$
    \begin{equation}
    \left(I(t)e^{2bt}\right)'=-(2a-2b)I(t)e^{2bt}+V(t)e^{2bt}.
    \end{equation} Integrating over $(0,t)$, we have
    \begin{equation}\label{Iest02}
    e^{2bt}\int_0^\infty x^2|u|^2dx+(2a-2b)
    \int_0^te^{2bs}\int_0^\infty x^2|u|^2dxds=\int_0^\infty x^2|u_0|^2dx+\int_0^tV(s)e^{2bs}ds.
    \end{equation}

    Differentiating $y(t)$, we have
    \begin{equation}
        \frac{d}{dt}y(t)=\frac{d}{dt}Im\int_0^\infty \bar{u}xu_xdx=
        Im\int_0^\infty (\bar{u}_txu_x+\bar{u}xu_{xt})dx.
    \end{equation}
Integrating by parts we obtain
    \begin{multline}Im \int_0^\infty \bar{u}xu_{xt}dx=-Im \int_0^\infty (\bar{u}x)_xu_tdx\\
    =-Im\int_0^\infty \bar{u}_xxu_tdx-Im \int_0^\infty \bar{u}u_tdx.
     \end{multline}
  Hence, \begin{equation}\label{yprime01}
        \frac{d}{dt}y(t)=2Im\int_0^\infty \bar{u}_t xu_xdx-Im\int \bar{u}u_tdx.
    \end{equation}
    The first term on the right hand side of \eqref{yprime01} is
  \begin{multline}
       2Im\int_0^\infty \bar{u}_t xu_xdx=2Im\int_0^\infty \left(i\bar{u}_{xx}-
       ik|u|^p\bar{u}-a\bar{u}\right)xu_x dx\\
       =2Re\int_0^\infty \bar{u}_{xx}xu_x dx-2Re\int_0^\infty kx|u|^p\bar{u}u_x dx
       -2aIm\int_0^\infty x\bar{u}u_xdx,
  \end{multline} where
      \begin{equation}
        2Re\int_0^\infty\bar{u}_{xx}xu_x dx=Re \int_0^\infty x(|u_x|^2)_x dx=
        -\int_0^\infty|u_x|^2dx
    \end{equation} and

     \begin{multline}
        -2Re\int_0^\infty kx|u|^p\bar{u}u_x dx =
        -\frac{2k}{p+2}Re \int_0^\infty x(|u|^{p+2})_x dx\\
        =\frac{2k}{p+2}\int_0^\infty|u|^{p+2}dx=\frac{2k}{p+2}\|u\|_{L^{p+2}(I)}^{p+2}.
    \end{multline}
The second term on the right hand side of \eqref{yprime01} is
    \begin{multline}\label{yprime02}
        -Im\int_0^\infty \bar{u}u_tdx=-Im\int_0^\infty\bar{u}(-iu_{xx}+ik|u|^pu-au)dx\\
        =Re\int_0^\infty\bar{u}u_{xx}dx-k\|u\|_{L^{p+2}(I)}^{p+2}=
        -\int_0^\infty |u_x|^2dx+\lambda|u(0,t)|^{r+2}-k\|u\|_{L^{p+2}(I)}^{p+2}.
    \end{multline}

      Combining \eqref{yprime01}-\eqref{yprime02}, we obtain
    \begin{equation}\label{yprime}
        \frac{d}{dt}y(t) = -2\|u_x\|^2-
        \frac{kp}{p+2}\|u\|_{L^{p+2}(I)}^{p+2}+\lambda|u(0,t)|^{r+2}-
        2aIm\int_0^\infty x\bar{u}u_xdx.
    \end{equation}

Multiplying \eqref{yprime} by $-4$ and rearranging the terms, we have
    \begin{equation}\label{Vprime}\frac{d}{dt}V(t)+2aV(t)=8\|u_x\|^2+
    \frac{4kp}{p+2}\|u\|_{L^{p+2}(I)}^{p+2}-4\lambda|u(0,t)|^{r+2}\equiv \theta_1(t).
    \end{equation}
    Adding $(2b-2a)V(t)$ to both sides of \eqref{Vprime}, multiplying the obtained relation
    by $e^{2bt}$ and integrating over the interval
    $(0,t)$, we obtain
    \begin{equation}\label{V01} V(t)e^{2bt}=V(0)+(2b-2a)\int_0^tV(s)e^{2bs}ds+
    \int_0^t\theta_1(s)e^{2bs}ds.
    \end{equation}
\end{proof}

Let $M=\max\{8,2p\}$ and $b=\frac{a(r+2)(4-M)}{4(r+2)-2M}$,
then $b<0$ since $r>\max\{2,p-2\}$, and moreover

    $$-M\left(\frac{a(r+2)-2b}{2a-2b}\right)\frac{2\lambda}{r+2}|u(0,t)|^{r+2}\ge -
    4\lambda|u(0,t)|^{r+2}.$$ On the other hand, $M\|u_x\|^2\ge 8\|u_x\|^2$ and
    $$
    M\frac{2k}{p+2}\|u\|_{L^{p+2}(I)}^{p+2}\ge \frac{4kp}{p+2}\|u\|_{L^{p+2}(I)}^{p+2}.
    $$

Therefore, $\theta_1(t)\le \theta(t)$, and by \eqref{e1negative} and
\eqref{V01},
    \begin{equation}
    V(t)e^{2bt}\le V(0)+(2b-2a)\int_0^tV(s)e^{2bs}ds,
    \end{equation} which can also be written as
    \begin{equation}
    \label{Vdt}\frac{d}{dt}\left(e^{(2a-2b)t}\int_0^tV(s)
    e^{2bs}ds\right)\le V(0)e^{(2a-2b)t}.
    \end{equation}
    Integrating \eqref{Vdt} over $(0,t)$, we obtain
    $$
    \int_0^tV(s)e^{2bs}ds\le \frac{1}{2a-2b}(1-e^{-(2a-2b)t})V(0).
    $$
    From this inequality, one obtains the blow-up of the solutions.
    Indeed, let
    $$
    z(t)\equiv e^{2bt}\int_0^\infty x^2|u|^2dx.$$ Then by
    \eqref{Iest02},
    $$z(t)\le \int_0^\infty x^2|u_0|^2dx+ \frac{1}{2a-2b}(1-e^{-(2a-2b)t})V(0).$$
    Hence,
 $$\lim_{t\rightarrow T}z(t)=0$$ where $T\equiv -
 \frac{1}{2a-2b}\ln\left( \frac{(2a-2b)\int x^2|u_0|^2dx+V(0)}{V(0)}\right).$
 We choose $u_0$ in such a way that $T>0$ by assumption \eqref{u0assmp}.
 Now, using the decay of the $L^2$ norm which was proved in Lemma
 \ref{massenergylemma}, we deduce the inequality

\begin{equation*}\|u(t)\|_{L^2(I)}^2=-2 Re\int_0^\infty x
u\bar{u}_xdx \le 2\|xu(x,t)\|_{L^2(I)}\cdot\|u_x(t)\|_{L^2(I)}.
\end{equation*}
The last inequality implies: $$ \|u_x(t)\|_{L^2(I)}\ge
\frac{\|u_0(x)\|_{L^2(I)}^2e^{-(2a-2b)t}}{z(t)}\rightarrow
\infty$$
as $t\rightarrow T.$
\subsection{Case $a=0$:}  In this section, we prove Theorem
\ref{MainThm01} for $a = 0$ by obtaining a nonlinear ordinary
differential inequality which yields blow-up of solutions.  The
proof follows by adapting the same argument in \cite{G} to our
model.
$$\frac{1}{2}\|u_x(t)\|_{L^2(I)}^2+\frac{k}{p+2}\|u(t)\|_{L^{p+2}(I)}^{p+2}=
E(0)+\frac{\lambda}{r+2}|u(0,t)|^{r+2}.$$
Then,
$$\frac{r+2}{2}\|u_x(t)\|_{L^2(I)}^2+\frac{k(r+2)}{p+2}\|u(t)\|_{L^{p+2}(I)}^{p+2}\le
\lambda|u(0,t)|^{r+2},$$ provided that $E(0)\le 0.$
$$y'(t)\ge (\frac{r-2}{2})\|u_x\|_{L^2(I)}^2+\frac{k(r-p+2)}{p+2}\|u(t)\|_{L^{p+2}(I)}^{p+2}.
$$ Then $y'(t)\ge \kappa\|u_x(t)\|_{L^2(I)}^2$ for some $\kappa>0$
provided that ${r>\max\{2,p-2\}}$. Therefore $y(t)>0$, since $y(0)>0.$
This means $I'(t)= -4y(t)\le 0.$ Hence, $I(t)\le I(0).$
By definition of $y(t)$, we have $|y(t)|\le \sqrt{I(0)}\|u_x\|_{L^2(I)}.$
Hence, $y'(t)\ge \kappa\frac{y^2(t)}{I(0)}$. Separating the variables and
integrating this differential inequality over the interval $(0,t)$, and using $y(0)>0$,
we get
$$\int_0^t \frac{dy}{y^2}=\int_0^t \frac{\kappa}{I(0)}ds\Rightarrow y(t)\ge
\frac{y(0)I(0)}{I(0)-\kappa y(0)t}.$$  That is to say,
$$\|u_x\|_{L^2(I)}\ge \frac{y(t)}{\sqrt{I(0)}}\ge \frac{y(0)\sqrt{I(0)}}{I(0)-
\kappa y(0)t}.$$  Hence, we deduce that $$\lim_{t\rightarrow T^-}\|u_x(t)\|_{L^2(I)}=
\infty$$ where $T\equiv \frac{I(0)}{\kappa y(0)}$.
\section{Critical Exponent and Exponential Decay Estimates}\label{SecCrt}
\subsection{Critical Exponent Conjecture}\label{Criticalexp}

It is not difficult to obtain uniform boundedness (in time variable) of the $H^1$
norm if $r<2$ for arbitrarily large initial data and if $r=2, p\le 4$ for small initial data.  In order to prove this, one can simply proceed as in \cite{AASKD} for $a=0$.  Regarding the damped situation ($a>0$), see {Section \ref{damprem}} below.  However, we expect that the situation in our model should be better than this due to the defocusing source term $k|u|^pu, k>0$.  We conjecture that if $p>4$, then one can control the $H^1$ norm of the solutions with arbitrarily large initial data, even if $2\le r < p-2$.  In addition, one should be able to control the $H^1$ norm with small data for $r\ge p-2$ whenever $p>4$.  More precisely, we have the following conjecture.
\begin{conjecture}The critical exponent for the nonlinear model \eqref{nlNeumannProb}
is $$r^*=\max\{2,p-2\}.$$
\end{conjecture}
One can try to use interpolation on $L^p-$spaces to obtain some partial results.
Let us assume $a=0$ for simplicity. Observe that
\begin{equation}\label{a0ineq}\|u_x(t)\|_{L^2(I)}^2+\frac{2k}{p+2}\|u(t)\|_{p+2}^{p+2}
\le |E(0)|+
\frac{2\lambda}{r+2}|u(0)|^{r+2}.\end{equation}  By $\epsilon-$Young's inequality
and Hölder's inequality,
\begin{multline}
  \frac{2\lambda}{r+2}|u(0,t)|^{r+2}=-\frac{2\lambda}{r+2}\int_0^\infty (|u|^{r+2})_xdx=
  -2\lambda Re\int_0^\infty  |u|^{r}u\bar{u}_xdx\\
  \le \epsilon\|u_x\|_{L^2(I)}^2+
  C_{\epsilon}\int_0^\infty |u|^{2r+2}dx=\epsilon\|u_x\|_{L^2(I)}^2+
  C_{\epsilon}\int_0^\infty |u|^{2r+2-\delta}|u|^{\delta}dx\\
  \le \epsilon\|u_x\|_{L^2(I)}^2+
  C_{\epsilon}\|u\|_{L^2(I)}^{\frac{\delta}{2}}\|u\|_{\frac{2(2r+2-\delta)}
  {2-\delta}}^{\frac{2-\delta}{2}},
\end{multline}
 where $\epsilon>0$ is fixed and can be chosen arbitrarily small.

If we choose $\delta=2-\frac{4r}{p}$, which is positive if $p>2r$,
use the mass identity (mass is conserved if $a=0$), and Hölder's
inequality again, then we obtain
\begin{multline*}|u(0,t)|^{r+2}\le
\epsilon\|u_x(t)\|_{L^2(I)}^2+
C_{\epsilon}\|u(t)\|_{L^2(I)}^{\frac{p-2r}{p}}\|u(t)\|_{p+2}^{\frac{2r}{p}}
\le\epsilon\|u_x(t)\|_{L^2(I)}^2\\
+
C_{\epsilon}\|u_0\|_{L^2(I)}^{\frac{(p+2)(p-2r)}{p(p+2)-2r}}+
\frac{2k\epsilon}{p+2}\|u(t)\|_{p+2}^{p+2}.
\end{multline*}
Using this in \eqref{a0ineq}, we get
$$(1-\epsilon)\left(\|u_x(t)\|_{L^2(I)}^2+\frac{2k}{p+2}\|u(t)\|_{p+2}^{p+2}\right)
\le |E(0)|+C_{\epsilon}\|u_0\|_{L^2(I)}^{\frac{(p+2)(p-2r)}{p(p+2)-2r}}.$$
Hence we have $\|u_x\|_{L^2(I)}\le C$ for some $C>0.$

One can improve the above analysis by involving the case ${r>2, r\ge \frac{p}{2}}$
under a smallness assumption on the initial data.  Indeed, by \eqref{a0ineq} and
\eqref{u0t}, we have
\begin{equation}\label{ineqa1}
  \|u_x(t)\|_{L^2(I)}^2\le \|u_0'\|_{L^2(I)}^2+
  \frac{2k}{p+2}\|u_0\|_{p+2}^{p+2}+
  \frac{2^{\frac{r+4}{2}}\lambda}{r+2}\|u_0\|_{L^2(I)}^{\frac{r+2}{2}}
  \|u_x\|_{L^2(I)}^{\frac{r+2}{2}}.
\end{equation}  If we set $\Phi(t)\equiv \|u_x(t)\|_{L^2(I)}^2$,
then \eqref{ineqa1} can be rewritten as
\begin{equation}\label{Phit}
\Phi(t)\le C_1+C_2\Phi(t)^{\sigma},
\end{equation}
where
\begin{equation}\label{C1C2}
C_1\equiv \|u_0'\|_{L^2(I)}^2+
\frac{2k}{p+2}\|u_0\|_{p+2}^{p+2},
C_2\equiv \frac{2^{\frac{r+4}{2}}\lambda}{r+2}\|u_0\|_{L^2(I)}^{\frac{r+2}{2}}
\end{equation} and $\sigma\equiv \frac{r+2}{4}>1.$  Since, $\Phi(0)\le C_1$,
then for sufficiently small $u_0$ one can have
$C_1C_2^{\frac{1}{\sigma-1}}\le \frac{\sigma-1}{\sigma^{\frac{\sigma}{\sigma-1}}}$,
we conclude that $\Phi(t)\le \frac{\sigma}{\sigma-1}C_1.$
For a justification of the smallness argument we carried out,
we use the following lemma.

\begin{lemma}[\cite{S68}]\label{smallnesslem} Suppose
$$\Phi(t)\le C_1+C_2\Phi(t)^\sigma, \ \ \forall t\in [0,T),$$ where $\Phi:[0,T)\rightarrow
\mathbb{R}$ is non-negative, continuous, $C_i>0$ $(i=1,2)$,
$\sigma>1,$ and $\gamma=\frac{1}{\sigma-1}$. If $\Phi(0)\le C_1$ and
$C_1C_2^\gamma\le (\sigma-1)\sigma^{-\gamma-1}$. Then
$$\Phi(t)\le
\frac{\sigma}{\sigma-1}C_1, \ \ \forall t\in [0,T).
$$
\end{lemma}

\subsection{Effect of Damping: Proof of Theorem \ref{MainThm02}}\label{damprem} Our analysis above shows that although it is more difficult to prove the blow-up result in the presence of the damping term $iau,a>0$, damping actually plays no particular role in the blow-up condition $r> \max\{2,p-2\}$.  This is analogous to the result in \cite{MT2}.  Nevertheless, damping may have a stabilizing effect in the case that global solutions exist.  See for example \cite{TM1}.  For our model this is easy to show in the case $r<2$, but is difficult to show if $2\le r<p-2$ whenever $p>4$, as in Section \ref{Criticalexp}.

Indeed, by Lemma \ref{massenergylemma}, we have

\begin{multline}\label{u02}|u(0,t)|^2=-\int_0^\infty (|u|^2)_xdx=
-2Re\int_0^\infty  u\bar{u}_xdx\\
\le 2\|u\|_{L^2(I)}\|u_x\|_{L^2(I)}\le 2\|u_0\|_{L^2(I)}e^{-at}\|u_x\|_{L^2(I)},
\end{multline} which implies \begin{equation}\label{u0t}|u(0,t)|^{r+2}\le 2^{\frac{r+2}{2}}\|u_0\|_{L^2(I)}^{\frac{r+2}{2}}e^{-a\frac{(r+2)}{2}t}\|u_x\|_{L^2(I)}^{\frac{r+2}{2}}.\end{equation}  Now, if $r<2$, then by $\epsilon-$Young's inequality, the right hand side of the above inequality is bounded by $$C_\epsilon e^{-a\mu t}+\epsilon \|u_x\|_{L^2(I)}^2$$ where $\epsilon, C_\epsilon>0$ (generic constants) and $\mu=\frac{2(r+2)}{2-r}$.  Observe that $\mu=2+\frac{4r}{2-r}>2.$  Multiplying identity \eqref{H1iden01b} by $e^{2at}$ and integrating over the time interval $(0,t)$,
\begin{equation}\label{Ee2at}
  E(t)e^{2at} = E(0) -\frac{2akp}{p+2}\int_0^t\|u(s)\|_{L^{p+2}(I)}^{p+2}e^{2as}ds+
  {\frac{2a\lambda r}{r+2}\int_0^t|u(0,t)|^{r+2}e^{2as}ds}, \\
\end{equation} which gives
\begin{multline}
  \|u_x\|_{L^2(I)}^2e^{2at}\le \frac{2\lambda}{r+2}|u(0,t)|^{r+2}e^{2at}-
  \frac{2k}{p+2}\|u(t)\|_{L^{p+2}(I)}^{p+2}e^{2at}\\
  +E(0)-\frac{2akp}{p+2}\int_0^t\|u(s)\|_{L^{p+2}(I)}^{p+2}e^{2as}ds+
  {\frac{2a\lambda r}{r+2}\int_0^t|u(0,t)|^{r+2}e^{2as}ds}\\
  \le \frac{2\lambda}{r+2}|u(0,t)|^{r+2}e^{2at}+E(0)+{\frac{2a\lambda r}{r+2}
  \int_0^t|u(0,t)|^{r+2}e^{2as}ds}\\
  \le C_\epsilon e^{a(2-\mu) t}+\epsilon \|u_x\|_{L^2(I)}^2e^{2at}+|E(0)|+
  \int_0^tC_\epsilon e^{a(2-\mu) s}ds+\epsilon\int_0^t\|u_x\|_{L^2(I)}^2e^{2as}ds
\end{multline} which implies
\begin{equation}
  \|u_x\|_{L^2(I)}^2e^{2at}\le C_\epsilon+\epsilon\int_0^t\|u_x\|_{L^2(I)}^2e^{2as}ds.
\end{equation} By Gronwall's lemma,
\begin{equation}
  \|u_x\|_{L^2(I)}^2e^{2at}\le C_\epsilon e^{\epsilon t}\Rightarrow
  \|u_x\|_{L^2(I)}^2\le C_\epsilon e^{-(2a-\epsilon)t}.
\end{equation}
Combining the above result with the $L^2$ decay (see Lemma \ref{massenergylemma}),
we obtain the following result.
\begin{proposition}[Stabilization I]\label{stab1} Let $a>0, r<2$ and $u$ be a
local solution of \eqref{nlNeumannProb} (see Remark \ref{localassmp}).
Then $u$ is global and  decays to zero exponentially fast in the
following sense:
 $$\|u(t)\|_{H^1(I)}^2\le Ce^{-(2a-\epsilon)t}, t\ge 0$$ where $\epsilon>0$
 is fixed and can be chosen arbitrarily small.
\end{proposition}

Regarding the powers $r\ge 2$, one can also obtain similar decay estimates, but only under a smallness assumption on the initial data for some values of $p$.

Let us start with the case $2\le r <\frac{p}{2}$.  By an argument similar to that in Section
\ref{Criticalexp}, we have the following estimate:
$$\frac{2\lambda}{r+2}|u(0,t)|^{r+2}\le \epsilon\|u_x\|_{L^2(I)}^2+
C_{\epsilon}\|u_0\|_{L^2(I)}^{\mu}e^{-a\mu t}+\frac{2k\epsilon}{p+2}\|u\|_{p+2}^{p+2}
$$ where \begin{equation}\label{mudef}\mu=\frac{(p+2)(p-2r)}{p(p+2)-2r}>0.
\end{equation}  By \eqref{Ee2at}, we have \begin{multline}
  \|u_x\|_{L^2(I)}^2e^{2at}+\frac{2k}{p+2}\|u(t)\|_{L^{p+2}(I)}^{p+2}e^{2at}\le
  \frac{2\lambda}{r+2}|u(0,t)|^{r+2}e^{2at}-\\
  +E(0)-\frac{2akp}{p+2}\int_0^t\|u(s)\|_{L^{p+2}(I)}^{p+2}e^{2as}ds+
  {\frac{2a\lambda r}{r+2}\int_0^t|u(0,t)|^{r+2}e^{2as}ds}\\
  \le \frac{2\lambda}{r+2}|u(0,t)|^{r+2}e^{2at}+E(0)+
  {\frac{2a\lambda r}{r+2}\int_0^t|u(0,t)|^{r+2}e^{2as}ds}\\
  \le |E(0)|+C_\epsilon e^{a(2-\mu) t}+\int_0^tC_\epsilon e^{a(2-\mu) s}ds\\
  +\epsilon \left(\|u_x\|_{L^2(I)}^2+\frac{2k}{p+2}\|u\|_{p+2}^{p+2}\right)e^{2at}+
  \epsilon\int_0^t\left(\|u_x\|_{L^2(I)}^2+\frac{2k}{p+2}\|u\|_{p+2}^{p+2}\right)e^{2as}ds.
\end{multline}

Observe that $$\int_0^tC_\epsilon e^{a(2-\mu) s}ds=
\frac{C_\epsilon}{a(2-\mu)} \left(e^{a(2-\mu) t}-1\right)\le
\frac{C_\epsilon}{a(2-\mu)}e^{a(2-\mu) t}.$$

Let us set $\Psi(t)\equiv \left(\|u_x\|_{L^2(I)}^2+
\frac{2k}{p+2}\|u\|_{p+2}^{p+2}\right)e^{2at}$, then the above inequality reads
$$(1-\epsilon)\Psi(t) \le \alpha(t) + \epsilon \int_0^t\Psi(s)ds$$
where $\alpha(t)\equiv |E(0)|+C_\epsilon\left(1+\frac{1}{a(2-\mu)}\right) e^{a(2-\mu) t}$.
Note that $\alpha$ is a non-decreasing function since $\mu<2.$
Now, by Gronwall's lemma we have
$$\Psi(t)\le \frac{1}{1-\epsilon}\alpha(t)\exp\left(\frac{\epsilon t}{1-\epsilon}\right),
$$ which gives
$$\|u_x\|_{L^2(I)}^2\le Ce^{-(a\mu-\epsilon)t}, t\ge 0.$$
This is a slower rate of decay than in Proposition \ref{stab1}.
Hence, we proved the following proposition.
\begin{proposition}[Stabilization II]\label{stab2} Let $a>0, 2\le r <\frac{p}{2}$ and $u$ be a local solution of \eqref{nlNeumannProb} (see Remark \ref{localassmp}).   Then $u$ is global and decays to zero exponentially fast in the following sense:
 $$\|u(t)\|_{H^1(I)}^2\le Ce^{-(a\mu-\epsilon)t}, t\ge 0,$$ where $\mu$ is given by \eqref{mudef}, and $\epsilon>0$ is fixed and can be chosen arbitrarily small.
\end{proposition}

Now, let us consider the case $r=2$ and $p\le 4$. Using \eqref{u0t},
we obtain \begin{equation}|u(0,t)|^4\le 2^2\|u_0\|_{L^2(I)}^{2}e^{-2at}\|u_x\|_{L^2(I)}^{2}.
\end{equation}  Now, by \eqref{Ee2at}, we have
\begin{multline}
  \|u_x\|_{L^2(I)}^2e^{2at}\le \lambda|u(0,t)|^{4}e^{2at}+|E(0)|+
  {a\lambda}\int_0^t|u(0,t)|^{4}e^{2as}ds\\
  \le |E(0)|+4\lambda\|u_0\|_{L^2(I)}^{2}e^{-2at}\|u_x\|_{L^2(I)}^{2}e^{2at}+{4a\lambda\|u_0\|_{L^2(I)}^{2}}\int_0^te^{-2as}\|u_x\|_{L^2(I)}^{2}e^{2as}ds.
\end{multline}

Now, if we assume $\|u_0\|_{L^2(I)}^{2}<\frac{1}{4\lambda}$, and since $e^{-2at}\le 1$, we have
\begin{equation}
  (1-4\lambda\|u_0\|_{L^2(I)}^{2})\|u_x\|_{L^2(I)}^2e^{2at}\le |E(0)|+{4a\lambda\|u_0\|_{L^2(I)}^{2}}\int_0^te^{-2as}\|u_x\|_{L^2(I)}^{2}e^{2as}ds,
\end{equation} from which it follows that
\begin{equation}
  \|u_x\|_{L^2(I)}^2e^{2at}\le \frac{|E(0)|}{1-4\lambda\|u_0\|_{L^2(I)}^{2}}+\frac{{4a\lambda\|u_0\|_{L^2(I)}^{2}}}{1-4\lambda\|u_0\|_{L^2(I)}^{2}}\int_0^te^{-2as}\|u_x\|_{L^2(I)}^{2}e^{2as}ds.
\end{equation}  Applying Gronwall's inequality to the above, we get :
\begin{multline}\label{decayr2}
  \|u_x\|_{L^2(I)}^2e^{2at}\le \frac{|E(0)|}{1-4\lambda\|u_0\|_{L^2(I)}^{2}}\exp\left(\frac{{4a\lambda\|u_0\|_{L^2(I)}^{2}}}{1-4\lambda\|u_0\|_{L^2(I)}^{2}}{\int_0^te^{-2as}ds}\right)\\
  \le \frac{|E(0)|}{1-4\lambda\|u_0\|_{L^2(I)}^{2}}\exp\left(\frac{{2\lambda\|u_0\|_{L^2(I)}^{2}}}{1-4\lambda\|u_0\|_{L^2(I)}^{2}}\right).
\end{multline}  Hence, there exists $C>0$ such that $\|u_x(t)\|_{L^2(I)}^2\le Ce^{-2at}$ for $t\ge 0$.  Therefore, we have proved the following result.
\begin{proposition}[Stabilization III]\label{stab3} Let $a>0, r=2, p\le 4$ and $u$ be a local solution of \eqref{nlNeumannProb} (see Remark \ref{localassmp}) such that $u_0$ is sufficiently small in $L^2$ sense.   Then $u$ is global and moreover $u$ decays to zero exponentially fast in the following sense:
 $$\|u(t)\|_{H^1(I)}^2\le Ce^{-2at}, t\ge 0.$$
\end{proposition}
Observe that the decay rate obtained in Proposition \ref{stab3} is faster than the decay rates in Proposition \ref{stab1} and Proposition \ref{stab2}.

Now, let us consider the case $r>2, r\ge \frac{p}{2}$.

By \eqref{Ee2at} and \eqref{u0t},
\begin{equation}
\|u_x\|_{L^2(I)}^2e^{2at}\le C_1+C_2\|u_x\|_{L^2(I)}^{\frac{r+2}{2}}e^{2at}+arC_2\int_0^te^{-a(\frac{r+2}{2})s}\|u_x\|_{L^2(I)}^{\frac{r+2}{2}}e^{2as}ds.
\end{equation} where $C_1$ and $C_2$ are given in \eqref{C1C2}.

 Let us define $\displaystyle S(t)=\sup_{[0,t]}\{\|u_x\|_{L^2(I)}^2e^{2as}\}$.   Then since $\frac{r+2}{4}>1$, we have
\begin{multline}
S(t)\le C_1+C_2S(t)^{\frac{r+2}{4}}+arC_2S(t)^{\frac{r+2}{4}}\int_0^te^{-a(\frac{r+2}{2})s}ds\\
\le C_1+\left(1+\frac{4r}{r+2}\right)C_2S(t)^{\frac{r+2}{4}}.
\end{multline}  By the same smallness argument in \eqref{Phit} or Lemma \ref{smallnesslem}, we obtain $$S(t)\le \frac{2(r+2)}{r-2}C_1.$$  Hence, we proved the following proposition,
\begin{proposition}[Stabilization IV]\label{stab4} Let $a>0, r>2, r\ge \frac{p}{2}$ and $u$ be a local solution of \eqref{nlNeumannProb} (see Remark \ref{localassmp}) such that $u_0$ is sufficiently small in $H^1\cap L^{p+2}$ sense.   Then $u$ is global and moreover $u$ decays to zero exponentially fast in the following sense:
 $$\|u(t)\|_{H^1(I)}^2\le Ce^{-2at}, t\ge 0.$$
\end{proposition}

\medskip

\end{document}